\documentclass[11pt, final]{amsart}
\usepackage[margin=1.3in]{geometry}
\usepackage{amsmath,amsthm}
\usepackage{amssymb,amsxtra,verbatim}
\usepackage{latexsym}
\usepackage{color}

\usepackage{enumitem}

\usepackage{graphicx}
\usepackage{tikz}
\usetikzlibrary{matrix}

\usepackage[cmtip,all]{xy}

\usepackage{xypic}
\usepackage[notcite, notref]{showkeys}

\usepackage[
	hyperindex,
	pdftex,
	pdfdisplaydoctitle,
	pdfpagemode=UseNone,
	breaklinks=true,
	extension=pdf,
	bookmarks=false,
	plainpages=false,
	colorlinks,
	linkcolor=blue,
	citecolor=black,
]{hyperref}
\usepackage{aliascnt}

\theoremstyle{plain}
\newtheorem*{main-theorem}{Main Theorem}

\numberwithin{equation}{section}

\theoremstyle{plain}
\newtheorem{theorem}{Theorem}[section]
\newtheorem{prop}[theorem]{Proposition}
\newtheorem{corollary}[theorem]{Corollary}

\newtheorem{lemma}[theorem]{Lemma}

\theoremstyle{definition}

\newtheorem{definition}[theorem]{Definition}

\newtheorem{remark}[theorem]{Remark}

\def\ra{\rightarrow}

\def\co{\colon\thinspace} 

\DeclareMathOperator{\spec}{Spec}

\DeclareMathOperator{\Sym}{Sym}

\def\Hn1{\mathcal{H}_{n,1}}

\DeclareMathOperator{\HH}{\mathrm{H}}

\def\B{\mathcal{B}}
\def\C{\mathcal{C}}

\def\O{\mathcal{O}}

\newcommand\Mg[1]{\overline{\mathcal{M}}_{#1}}

\def\QQ{\mathbb{Q}}
\def\ZZ{\mathbb{Z}}

\def\PP{\mathbb{P}}
\def\ZZ{\mathbb{Z}}

\def\CC{\mathbb{C}}
\def\GG{\mathbb{G}}

\newcommand{\M}{\overline{M}}

\DeclareMathOperator\SL{SL}
\DeclareMathOperator\GL{GL}
\def\Grass{\operatorname{Grass}}

\def\CS{\operatorname{CoSyz}}

\def\Syz{\operatorname{Syz}}

\newcommand{\gitq}{/\hspace{-0.25pc}/}

\begin{document}

\title{Toward GIT stability of syzygies of canonical curves}
\begin{abstract}
We introduce the problem of GIT stability for syzygy points of canonical curves 
with a view toward a GIT construction of the canonical model of $\M_g$. 
As the first step in this direction, we prove semi-stability of the $1^{st}$ syzygy point 
for a general canonical curve of odd genus.
\end{abstract}
\author{Anand Deopurkar} 
\author{Maksym Fedorchuk} 
\author{David Swinarski}

\address[Deopurkar]{Department of Mathematics\\
Columbia University\\
2990 Broadway\\
New York, NY 10027}
\email{anandrd@math.columbia.edu}

\address[Fedorchuk]{Department of Mathematics \\ Boston College \\
140 Commonwealth Avenue \\
Chestnut Hill, MA 02467}
\email{maksym.fedorchuk@bc.edu}

\address[Swinarski]{Department of Mathematics\\
Fordham University\\
113 W 60th Street\\
New York, NY 10023}
\email{dswinarski@fordham.edu}

\thanks{The second author was partially supported by NSF grant DMS-1259226.}

\maketitle

\section{Introduction}
\label{S:introduction}

We revisit the problem of studying syzygies of canonically embedded
rational ribbons originally posed by Bayer and Eisenbud in
\cite{BE}. Their motivation for studying ribbons was in the context of
Green's conjecture for smooth canonical curves. Our motivation is
different, but related. Namely, we are interested in GIT stability of the
syzygies of canonically embedded curves 
as the means to the eventual goal of giving a GIT construction of the canonical model of $\M_g$.

The problem of GIT stability for the syzygy points of canonical
curves has origins in the log minimal model program for the moduli
space of stable curves.  Introduced by
Hassett and Keel, this program aims to construct log canonical models of $\M_g$ in a way
that allows modular interpretations of these models as moduli spaces
of stacks of increasingly more singular curves
\cite{hassett_genus2}. The log canonical divisors on (the stack)
$\Mg{g}$ considered in this program are
$K_{\Mg{g}}+\alpha\delta=13\lambda-(2-\alpha)\delta$ for $\alpha \in [0,1]\cap \QQ$.  
The work done so far suggests
that we can construct some of these models as GIT quotients of
spaces of $n$-canonically embedded curves. This is already evidenced in the
work of Gieseker \cite{gieseker} and Schubert \cite{schubert}, who analyzed the cases of $n\geq 5$ and $n = 3$, respectively.
Recent work of Hassett and Hyeon \cite{HH1,HH2} extends the
GIT analysis to $n = 2$ and constructs the first two log canonical models of $\M_g$ corresponding to $\alpha > \frac23$.
Subsequent work along this direction suggests that the case of $n = 1$ and the use of
finite Hilbert points would yield log canonical models corresponding
to the values of $\alpha$ down to $\alpha=\frac{g+6}{7g+6}$
\cite{AFS-stability,fedorchuk-jensen}.

The ultimate goal of the Hassett--Keel program is to reach $\alpha =
0$, which corresponds to the canonical model of $\M_g$. To go beyond
$\alpha = \frac{g+6}{7g+6}$ and indeed down to $\alpha = 0$, Se\'{a}n
Keel suggested that one should construct birational models of $\M_g$
as GIT quotients using the syzygies of canonically embedded curves. In
this paper, we make the first step toward this goal. We set up the GIT
problem for the syzygy points and  prove a generic semi-stability
result for the $1^{st}$ syzygies in odd genus.
\begin{main-theorem}[\autoref{T:main-stability}]
A general canonical curve of odd genus has a semi-stable $1^{st}$ syzygy point.
\end{main-theorem}

Our strategy for proving generic stability of syzygy points follows
that of \cite{AFS-stability} for proving generic stability of finite
Hilbert points. Namely, we prove the semi-stability of the $1^{st}$
syzygy point of a special singular curve---the balanced ribbon---by a method of \cite{morrison-swinarski}.

\subsection*{Outline of the paper} 
In Section~\ref{S:syzygies}, we define syzygy points of a
canonically embedded curve and give a precise statement of our main result. In
Section~\ref{S:ribbon}, we recall some preliminary results about balanced ribbons.
In Sections~\ref{S:stability} and~\ref{S:monomial-bases}, we
prove the main theorem. 
More precisely, Section~\ref{S:stability} contains the proof assuming the existence of suitable
bases for the space of (co)syzygies of the balanced ribbon and the more technical Section~\ref{S:monomial-bases}
is devoted to the construction of these bases.

\subsection*{Acknowledgements} 
We learned the details of Se\'{a}n Keel's idea to use syzygies as 
the means to construct the canonical model of $\M_g$ from a talk given by
Gavril Farkas at the AIM workshop \emph{Log minimal model program for moduli spaces} 
held in December 2012.  This paper grew out of our
attempt to implement the roadmap laid out in that talk.  We are
grateful to AIM for the opportunity to meet. The workshop participants of the working group on syzygies,
among them David Jensen, Ian Morrison, Anand Patel, and the present
authors, verified by a computer computation our main result for
$g=7$. This computation motivated us to search for a proof in the general case.

\section{Syzygies of canonical curves}
\label{S:syzygies}

In this section, we recall the notions of Koszul cohomology necessary to set-up
the GIT problems for all syzygies of canonical curves. We refer to \cite{green} and
\cite{aprodu-farkas-survey} for a complete treatment of Koszul
cohomology and a detailed discussion of Green's conjecture.

\subsection{Koszul cohomology}
\label{S:koszul}
Let $C$ be a Gorenstein curve with a very ample dualizing sheaf
$\omega_C$. Associated to $C$ is the Koszul complex
\begin{equation*}
  \bigwedge^{p+1}\HH^0(\omega_C)\otimes \HH^0(\omega_C^{q-1}) 
  \xrightarrow{f_{p+1,q-1}}
  \bigwedge^{p}\HH^0(\omega_C)\otimes \HH^0(\omega_C^{q}) 
  \xrightarrow{f_{p,q}}
  \bigwedge^{p-1}\HH^0(\omega_C)\otimes \HH^0(\omega_C^{q+1})
\end{equation*}
where the differentials $f_{p,q}$ are given by
\[f_{p,q}(x_0\wedge x_1\wedge \cdots\wedge x_{p-1}\otimes y)
=\sum_{i=0}^{p-1} (-1)^i x_0\wedge \cdots \wedge \widehat{x_i}\wedge \cdots \wedge x_{p-1}\otimes x_i y.
\]
The Koszul cohomology groups are 
\[K_{p,q}(C):=\ker f_{p,q} \big{/} \operatorname{im} f_{p+1,q-1}.\]

We say that $C$ satisfies property $(N_{p})$ if $K_{i,q}(C) = 0$ for
all $(i,q)$ with $i \leq p$ and $q \geq 2$. Property $(N_0)$ means
that the natural maps $\Sym^m \HH^0(\omega_C) \to \HH^0(\omega_C^m)$
are surjective for all $m$. Property $(N_p)$ for $p \geq 1$ means, in
addition, that the ideal of $C$ in the canonical embedding is
generated by quadrics and the syzygies of order up to $p$ are linear.

Set
\[
\Gamma_p(C) := \left(\bigwedge^{p+1} \HH^0(\omega_C) \otimes \HH^0(\omega_C)\right) \Big{/} \ \bigwedge^{p+2} \HH^0(\omega_C).
\] 

The first four terms of the Koszul complex in degree $p+2$ give the exact sequence
\[ 0 \to K_{p+1,1}(C) \to \Gamma_p(C)  \to \ker f_{p,2} \to K_{p,2}(C) \to 0.\]

\begin{definition}\label{D:syzygies}
  We define the \emph{space of $p^{th}$ order linear syzygies} of $C$
  as the subspace of $\Gamma_p(C)$ given by
  \[ \Syz_p(C) := K_{p+1,1}(C).\] Suppose $C$ satisfies property
  $(N_p)$ so that $K_{p,2}(C)=0$. We define the \emph{space of $p^{th}$ order linear
    cosyzygies} of $C$ as the quotient space of $\Gamma_p(C)$ given by
  \[ \CS_p(C) := \ker f_{p,2}.\]
\end{definition}

We relate the above definition to the definition of syzygies in terms of the homogeneous ideal of $C$. Let 
\[I_m(C) = \ker \left(\Sym^m \HH^0(\omega_C) \to \HH^0(\omega_C^m)\right)\]
be the degree $m$ graded piece of the homogeneous ideal of $C$. Assume that $C$ satisfies property $(N_0)$. 
Then the space of $p^{th}$ order linear syzygies among the defining quadrics of $C$ is taken to be the kernel of the map
\[ \bigwedge^p \HH^0(\omega_C) \otimes I_2(C) \stackrel\alpha\longrightarrow \bigwedge^{p-1} \HH^0(\omega_C) \otimes I_3(C).\]
The following lemma shows that this kernel is isomorphic to $K_{p+1,1}(C)$, and thus justifies Definition~\ref{D:syzygies}.
\begin{lemma}\label{L:proj-normal-Koszul}
  Assume that $C$ satisfies $(N_0)$. Then we have a natural isomorphism
  \[
  \ker \alpha \simeq K_{p+1,1}(C).
  \]
\end{lemma}
\begin{proof}
  We compare the Koszul complexes associated to the coordinate ring of the projective space $\PP\HH^0(\omega_C)$, 
  the coordinate ring of $C$, and the homogeneous ideal of $C$. For brevity, we write $\omega$ instead of $\omega_C$. 
  Consider the commutative diagram
  \begin{equation*}
    \begin{tikzpicture}
      \matrix (M)
      [
      matrix of nodes,
      column sep=1cm,
      row sep=.5cm,
      nodes={execute at begin node=$,
        execute at end node=$}
      ]
      {
        0 & \bigwedge^{p+2}\HH^0(\omega) & \bigwedge^{p+2}\HH^0(\omega)\\
        0 & \bigwedge^{p+1}\HH^0(\omega) \otimes \HH^0(\omega) & \bigwedge^{p+1}\HH^0(\omega) \otimes \HH^0(\omega)\\
        \bigwedge^p \HH^0(\omega) \otimes I_2(C) & \bigwedge^p\HH^0(\omega) \otimes \Sym^2\HH^0(\omega) & \bigwedge^p\HH^0(\omega) \otimes \HH^0(\omega^2) \\
        \bigwedge^{p-1} \HH^0(\omega) \otimes I_3(C) & \bigwedge^p\HH^0(\omega) \otimes \Sym^3\HH^0(\omega) & \bigwedge^p\HH^0(\omega) \otimes \HH^0(\omega^3) \\
      };
      \foreach \i in {1,2,3,4}
      \draw [->] 
      (M-\i-1) edge (M-\i-2) 
      (M-\i-2) edge (M-\i-3)
      ;
      \foreach \j in {1,2,3}
      \draw [->]
      (M-1-\j) edge (M-2-\j)
      (M-2-\j) edge (M-3-\j)
      (M-3-\j) edge (M-4-\j);
      \draw [->] (M-3-1) edge node [right] {$\alpha$} (M-4-1);
    \end{tikzpicture}
  \end{equation*}
  The rows form short exact sequences due to property $(N_0)$. The middle column is exact since it is the Koszul complex associated to 
  $\O(1)$ of the projective space $\PP\HH^0(\omega)$. The right column is the Koszul complex of $C$. Taking the long exact sequence
   associated to the short exact sequence of the columnwise complexes, and using the exactness of the middle column, we get the isomorphism
  \[ \ker \alpha \simeq K_{p+1,1}(C).\]
\end{proof}

\subsection{Syzygy points}
\label{S:syzygy-points}
Suppose $C$ satisfies property $(N_p)$. Then the Koszul complex in degree $p+2$
\[ 0\to \bigwedge^{p+2} \HH^0(\omega_C) \to \dots 
\to \bigwedge^i \HH^0(\omega_C)\otimes \HH^0(\omega_C^{p+2-i}) \to \dots \to \HH^0(\omega_C^{p+2}) \to 0.\]
is exact everywhere except at the second non-zero term, where the cohomology group is $K_{p+1,1}(C)$. We can thus readily compute that
\begin{equation}\label{eqn:dimensions}
  \begin{split}
    \dim \CS_p(C) &= (3g-2p-3) {{g-1} \choose p} \text{, and} \\
    \dim \Gamma_p(C) &= g{g \choose {p+1}} - {g \choose {p+2}}.
  \end{split}
\end{equation}

\begin{definition}\label{D:syzygy-points}
  Suppose $C$ satisfies property $(N_p)$. We define the \emph{$p^{th}$
    syzygy point of $C$} to be the quotient of $\Gamma_p(C)$ given by
  \[
\left[  \Gamma_p(C) \ra \CS_p(C) \ra 0\right],
  \]
 and interpreted as a point in the Grassmannian $\Grass\left((3g-2p-3){{g-1} \choose p}, \Gamma_p(C)\right)$. 
\end{definition}
Abusing notation, we use $\CS_p(C)$ to denote both the vector space itself and the point in 
$\Grass\left((3g-2p-3){{g-1} \choose p}, \Gamma_p(C)\right)$ that it represents. 
Observe that the $0^{th}$ syzygy point is simply the $2^{nd}$ Hilbert point. 

For which curves is the $p^{th}$ syzygy point defined? According to a
celebrated conjecture, a smooth canonical curve $C$ satisfies $(N_p)$
if and only if $p$ is less that the Clifford index of $C$. Formulated
by Green in \cite{green}, this conjecture remains open in its full
generality.  It is known to be true, however, for a large class of
curves. Voisin proved that general canonical curves on $K3$ surfaces
satisfy Green's conjecture
\cite{voisin-green-even,voisin-green-odd}. More recently, Aprodu and
Farkas proved the conjecture for all smooth curves on $K3$ surfaces
\cite{aprodu-farkas-K3}. In particular, the $p^{th}$ syzygy point of a generic curve of genus $g$ is
defined for all $p < \lfloor g/2
\rfloor$.

\begin{definition}
We define $\overline \Syz_p$ to be the closure in
$\Grass\left((3g-2p-3){{g-1} \choose p}, \Gamma_p(C)\right)$ 
of the locus of $p^{th}$ syzygy points of canonical curves satisfying property $(N_p)$. 
\end{definition}

Consider the group $\SL_g \simeq \SL (\HH^0(\omega_C))$. Its natural action on $\HH^0(\omega_C)$ induces the action on the vector space 
$\Gamma_p(C)$, the Grassmannian $\Grass\left((3g-2p-3){{g-1} \choose p}, \Gamma_p(C)\right)$, and finally on 
the subvariety $\overline \Syz_p$. The Pl\"ucker line bundle on the Grassmannian comes with a natural $\SL_g$ linearization, 
and so does its restriction to $\overline \Syz_p$. 
A candidate for the \emph{$p^{th}$ syzygy model} of $\overline M_g$ is thus the GIT quotient
\[ \overline \Syz_p \gitq \SL_g.\]

Our main theorem shows that this quotient is non-empty for $p=1$ and odd $g$.
\begin{theorem}\label{T:main-stability}
  A general canonical curve of odd genus has a semi-stable $1^{st}$
  syzygy point.
\end{theorem}
We prove this theorem in Section \ref{S:stability}.

\section{The balanced canonical ribbon}
\label{S:ribbon}
We prove Theorem~\ref{T:main-stability} by explicitly writing down a
semi-stable point in $\overline \Syz_1$.  This point corresponds to
the syzygies of the balanced ribbon. Our exposition of its properties
closely follows \cite{AFS-stability} where the semi-stability of
Hilbert points of this ribbon was established.  Nevertheless, we
recall the necessary details for the reader's convenience.
 
Let $g = 2k + 1$. The \emph{balanced ribbon} of genus $g$ is the
scheme $R$ obtained by identifying $U := \spec \CC[u,
\epsilon]/(\epsilon^2)$ and $V := \spec \CC[v,\eta]/(\eta^2)$ along $U
\setminus \{0\}$ and $V \setminus \{0\}$ via the isomorphism
\begin{equation}\label{E:gluing}
\begin{aligned}
u &\mapsto v^{-1} - v^{-k-2} \eta, \\
\epsilon &\mapsto v^{-g-1} \eta.
\end{aligned}
\end{equation}
The scheme $R$ is an example of a \emph{rational ribbon}. While our proofs use only 
the balanced ribbon, we refer the reader to \cite{BE} for a more extensive
study of ribbons in general. 

Being a Gorenstein curve, $R$ has a
dualizing line bundle $\omega$, generated by $\frac{du \wedge d\epsilon}{\epsilon^2}$ on $U$, 
and by $\frac{dv \wedge d\eta}{\eta^2}$ on $V$. Since $\omega$ is very ample by
\cite[Lemma 3.2]{AFS-stability}, the global sections of $\omega$ embed $R$ as an arithmetically 
Gorenstein curve in $\PP^{g-1}$. As a result, we have $K_{i,q}(R)=0$ for all $q \geq 3$ and $i \leq g-3$.
In particular, property $(N_p)$ is equivalent to $K_{p,2}(R)=0$ \cite{ein}.

The balanced ribbon $R$ admits a $\GG_m$-action, given by
\begin{alignat*}{2}
  t \cdot u &\mapsto tu, &\quad t \cdot \epsilon &\mapsto t^{k+1} \epsilon \, , \\
  t \cdot v &\mapsto t^{-1}v, &\quad t \cdot \eta &\mapsto t^{-k-1}
  \eta.
\end{alignat*}
This action induces $\GG_m$-actions on $\HH^0(R, \omega^m)$ for all $m$. The
next two propositions describe these spaces along with their
decomposition into weight spaces.
\begin{prop}\label{P:mult-free}
  A basis for $\HH^0(R, \omega)$ is given by $x_0, \dots, x_{2k}$,
  where the $x_i$'s restricted to $U$ are given by
 \[
 x_i = \begin{cases}
   u^i \ \frac{du \wedge d\epsilon}{\epsilon^2}& \text{ if $0 \leq i \leq k$}\\
   \left(u^i + (i-k)u^{i-k-1} \epsilon\right)\ \frac{du \wedge
     d\epsilon}{\epsilon^2} & \text{ if $k < i \leq 2k$},
 \end{cases}
 \]
 and where $x_i$ is a $\GG_m$-semi-invariant of weight $i-k$. 
 In particular, $\HH^0(R, \omega)$ splits as a direct sum of $g$ distinct $\GG_m$ weight-spaces of weights $-k,\dots, k$.
\end{prop}
\begin{proof}
That $x_i$'s form a basis follows from \cite[Theorem 5.1]{BE}. The statement about the weights is obvious.
\end{proof}

\begin{remark}[$\ZZ_2$-symmetry]\label{R:symmetry}
Observe that $R$ has a $\ZZ_2$-symmetry given by the isomorphism $V\simeq U$ defined by 
$u\leftrightarrow v$ and $\epsilon \leftrightarrow \eta$ and commuting with the gluing isomorphism \eqref{E:gluing}.
The $\ZZ_2$-symmetry exchanges $x_i$ and $x_{2k-i}$. 
\end{remark}

The following basic observation helps in dealing with higher powers of $\omega$.

\begin{lemma}[Ribbon Product Lemma]
\label{lem:rpl}
Let $0 \leq i_1, \dots, i_m \leq 2k$ be such that $i_1, \dots, i_\ell \leq k$ and $i_{\ell+1}, \dots, i_m > k$. On $U$, we have
\[ x_{i_1} \cdots x_{i_m} =\left(u^a + (a-b)u^{a-k-1}\epsilon\right)\ \left(\frac{du \wedge d\epsilon}{\epsilon^2}\right)^m,\]
where
 \begin{align*}
   a &= i_1 + \dots + i_m \ , \\
   b &= i_1 + \dots + i_\ell + k(m-\ell).
 \end{align*}
\end{lemma}
\begin{proof}
This is simply \cite[Lemma 3.4]{AFS-stability} in our notation.
\end{proof}

\begin{definition}
\label{D:u-degree}
The \emph{$u$-weight} (or \emph{$u$-degree})
of a monomial $x_{i_1}\cdots x_{i_m}$ is the sum $i_1 + \dots + i_m$.
Note that the $u$-weight of $x_{i_1}\cdots x_{i_m}$ 
equals to the $\GG_m$-weight of $x_{i_1}\cdots x_{i_m}$ plus $km$.
\end{definition}

\begin{prop}\label{thm:weight_decomposition_wm}
 Let $m \geq 2$. Let $\HH^0(R, \omega^m)_d$ be the weight-space of $\HH^0(R, \omega^m)$ of 
 $u$-weight $d$. Then
 \[ \dim\HH^0(R, \omega^m)_d =   \begin{cases}
   1 &\text{if $0 \leq d \leq k$,}\\
   2 &\text{if $k < d < 2km-k$,}\\
   1 &\text{if $2km-k \leq d \leq 2km$.}
 \end{cases}
 \]
 Moreover, the map $\Sym^m\HH^0(R, \omega) \to\HH^0(R, \omega^m)$ is surjective.
\end{prop}

\begin{proof}
 Using the generator $(\frac{du\wedge d\epsilon}{\epsilon^2})^m$ of $\omega^m$ on $U$, 
 let us identify the sections of $\omega^m$ on $U$ with the elements of $\CC[x,\epsilon]/(\epsilon^2)$.  
 Consider the following $(2m-1)(g-1)$ sections of $\omega^m$ on $U$:
 \begin{equation}\label{eq:sections_of_mw}
   \{u^i\}_{i=0}^{2mk-k-1}, \quad \{u^i+(i-mk)u^{i-k-1}\epsilon \}_{i=k+1}^{2mk}.
 \end{equation}
 We claim that these sections are in the image of $\Sym^m \HH^0(R, \omega)$. Indeed, for $0 \leq i \leq k$, the monomial $x_0^{m-1} x_{i}$ 
 restricts to $u^i$. For $2mk-k \leq i \leq 2mk$, the monomial $x_{2k}^{m-1}x_{i+2k-2mk}$ restricts to $u^i + (i-mk)u^{i-k-1}$. For $k < i < 2mk-k$, 
 it suffices to exhibit two monomials $x_{i_1}\cdots x_{i_m}$ with $i_1 + \dots + i_m = i$ whose restrictions to $U$ are linearly independent. 
This is easy to do using Lemma~\ref{lem:rpl}; we leave this to the reader.

We conclude that the sections listed in \eqref{eq:sections_of_mw} extend to global sections of $\omega^m$. By construction, these global sections 
are in the image of $\Sym^m\HH^0(R, \omega)$. Since these sections are linearly independent and their number equals $h^0(\omega^m)$, they 
form a basis of $\HH^0(\omega^m)$. We conclude that $\Sym^m\HH^0(R, \omega) \to\HH^0(R, \omega^m)$ is surjective. 
The sections $u^i (\frac{du\wedge d\epsilon}{\epsilon^2})^m$ are eigenvectors of $\GG_m$ with weights $-km, \dots, km-k-1$. The sections 
$(u^i+(i-mk)u^{i-k-1}\epsilon)(\frac{du\wedge d\epsilon}{\epsilon^2})^m$ are eigenvectors of $\GG_m$ with weights $-km+k+1, \dots, km$. 
Combining the two, we get the dimensions of the weight spaces.
\end{proof}

The following is immediate from Proposition~\ref{thm:weight_decomposition_wm}.
\begin{corollary}\label{thm:what_makes_a_basis}
 Let $\B$ be a set of monomials of degree $m$ in the variables $x_0, \dots, x_{2k}$. Its image in $\HH^0(R, \omega^m)$ forms a basis if and only if
 \begin{enumerate}
 \item For $0 \leq d \leq k$ and $2km - k \leq d \leq 2km$, $\B$ contains one monomial of $u$-weight $d$.
 \item For $k < d < 2km-k$, $\B$ contains two monomials of $u$-weight $d$ and these two monomials 
 are linearly independent in $\HH^0(R, \omega^m)$.
 \end{enumerate}
\end{corollary}


We recall the following result:
\begin{prop}
\label{P:quadratic-bases}
The following are bases of $\HH^0(R, \omega^2)$: 
\begin{align}\label{eq:Bases_2w}
 \mathcal B^+ &:= \{x_0x_i\}_{i=0}^{2k} \cup \{x_kx_i\}_{i=1}^{2k-1} \cup \{x_{2k}x_i\}_{i=1}^{2k}\\
 \mathcal B^- &:= \{x_i^2\}_{i=0}^{2k} \cup \{x_ix_{i+1}\}_{i=0}^{2k-1} \cup \{x_ix_{i+k}\}_{i=1}^{k-1} \cup \{x_ix_{i+k+1}\}_{i=0}^{k-1}.
\end{align}
Both $\B^{+}$ and $\B^{-}$ are symmetric with respect to the $\ZZ_2$-symmetry of $R$
and consist of $\GG_m$-semi-invariant sections. 
The breakdown of $\mathcal B^+$ by $u$-weight in the range $0 \leq d \leq 2k$ is:
\begin{alignat*}{2}
 &x_0 x_d &\quad&\text{for $0 \leq d \leq k$}  \\
 &x_0 x_d,\ x_k x_{d-k} &\quad&\text{for $k < d \leq 2k$} 
\end{alignat*}
The breakdown of $\mathcal B^-$ by $u$-weight in the range $0 \leq d \leq 2k$ is:
\begin{alignat*}{2}
 &x_{\lfloor d/2 \rfloor} x_{\lceil d/2 \rceil} &\quad&\text{for $0 \leq d \leq k$}  \\
 &x_{\lfloor d/2 \rfloor} x_{\lceil d/2 \rceil},\ x_{\lfloor (d-k)/2 \rfloor}x_{\lceil (d+k)/2 \rceil} &\quad&\text{for $k < d \leq 2k$}.
\end{alignat*}
The breakdown in the range $2k\leq d\leq 4k$ is obtained by using the $\ZZ_2$-symmetry.  
\end{prop}
\begin{proof}
The fact that $\B^{+}$ and $\B^{-}$ are bases of $\HH^0(R, \omega^2)$ is the content of \cite[Lemma 4.3]{AFS-stability}.
The weight decomposition statement is obvious. 
\end{proof}

We record a simple observation about expressing arbitrary quadratic monomials in $\HH^0(R, \omega^2)$ 
in terms of the monomials of $\mathcal B^-$
(it will be used repeatedly in Section \ref{second-basis}):
\begin{lemma}[Quadratic equations]\label{L:quadratic}
 Consider $0 \leq i \leq j \leq 2k$ and set $d = i+j$. 
 Then in $\HH^0(R, \omega^2)$ we have a relation
 \[ x_i x_j = \lambda x_{\lfloor d/2 \rfloor} x_{\lceil d/2 \rceil} + \mu x_{\lfloor (d-k)/2 \rfloor}x_{\lceil (d+k)/2 \rceil},\]
 where $\lambda$ and $\mu$ are uniquely determined rational numbers satisfying:
 \begin{enumerate}
 \item  $\lambda + \mu = 1$,
 \item if $j\leq k$ or $i\geq k$, then $\mu=0$,
 \item if $j-i = k$ or $j-i = k+1$, then $\lambda = 0$;   
 \item if $j-i < k$, then $\lambda, \, \mu > 0$;
 \item if $j-i > k+1$, then $\lambda < 0$, $\mu > 0$.
 \end{enumerate}
\end{lemma}
\begin{proof} The existence and uniqueness of the relation follows from Proposition \ref{P:quadratic-bases}.
We now establish the claims about the coefficients for $k < d < 3k$, the remaining cases being clear.
 By the $\ZZ_2$-symmetry, we may take $k < d \leq 2k$. If $j\leq k$, the statement is clear. If $j>k$, then
 \begin{align*}
   x_{\lfloor d/2 \rfloor} x_{\lceil d/2 \rceil} &= u^d,\\
   x_{\lfloor (d-k)/2 \rfloor} x_{\lceil (d+k)/2 \rceil} &= u^d + \lceil (d-k)/2 \rceil u^{d-k}\epsilon,\\
   x_i x_j &= u^d + (j-k)u^{d-k} \epsilon.
 \end{align*}
Now, (1) follows from equating the coefficients of $u^d$. If $j-i = k$ or $j-i = k+1$, then 
$(i,j) = (\lfloor (d-k)/2 \rfloor, \lceil (d+k)/2 \rceil)$; so (3) follows. Finally, (4) and (5)
follow from equating the coefficients of $u^{d-k}\epsilon$ and observing that if $j - i < k$, then 
$j-k < \lceil (d-k)/2 \rceil$, and if $j-i > k+1$, then $j-k > \lceil(d-k)/2 \rceil$.
\end{proof}

\section{Semi-stability of the $1^{st}$ syzygy point}
\label{S:stability}
In this section, we prove that the balanced canonical ribbon $R$ has 
semi-stable $1^{st}$ syzygy point, while relegating the key technical constructions 
to the next section. 

Let $R$ be the balanced canonical ribbon introduced in the previous section. We abbreviate
$\HH^0(R,\omega^m)$ as $\HH^0(\omega^m)$. We also set
\[ \Gamma := \left(\bigwedge^2 \HH^0(\omega) \otimes \HH^0(\omega)\right)
\Big{/} \ \bigwedge^3 \HH^0(\omega).\] 
For $x,y,z\in \HH^0(\omega)$, we call the image of $(x\wedge y)\otimes z$ in $\Gamma$
a \emph{cosyzygy}. By a slight abuse of notation, we use the same notation for $(x\wedge y)\otimes z$ 
and its image in $\Gamma$.  
With this convention, the only linear relations among cosyzygies in $\Gamma$ are 
\[ (x\wedge y)\otimes z + (y\wedge z)\otimes x +  (z\wedge x)\otimes y= 0.\] 

For the $1^{st}$ syzygy point, the relevant strand of the Koszul complex
is
\[ 0 \to \bigwedge^{3} \HH^0(\omega) \to \bigwedge^2 \HH^0(\omega) \otimes
\HH^0(\omega) \xrightarrow{f_{2,1}} \HH^0(\omega) \otimes \HH^0(\omega^2)
\xrightarrow{f_{1,2}} \HH^0(\omega^3) \to 0.\] 
By Definition \ref{D:syzygy-points}, the $1^{st}$ syzygy point of $R$ is
well-defined if and only if $K_{1,2}(R) = 0$ if and only if the map $\Gamma \to \ker f_{1,2}$
induced by the above complex is surjective. Before proving that $K_{1,2}(R) = 0$, we make a
definition.
\begin{definition}\label{D:monomial-basis}
  A set $\mathcal{C}=\{(x_a \wedge x_b) \otimes x_c\}_{(a,b,c)\in S}
  \subset \Gamma$ is called a \emph{monomial basis of cosyzygies} if
$\{f_{2,1}\bigl((x_a \wedge x_b) \otimes x_c\bigr)\}_{(a,b,c)\in S}$ form a
  basis of $\ker f_{1,2}$.
\end{definition}
Note that if $\Gamma \to \ker f_{1,2}$ is surjective, then by
\eqref{eqn:dimensions} we have
\[ \dim \ker f_{1,2} = (3g-5)(g-1).\] Therefore, a set
$\mathcal{C}=\{(x_a \wedge x_b) \otimes x_c\}_{(a,b,c)\in S} \subset
\Gamma$ is a monomial basis of cosyzygies if and only if the following two conditions are satisfied
\begin{enumerate}
\item $\C$ has $(3g-5)(g-1)$ elements,\\
\item $\left\{f_{2,1}\bigl((x_a \wedge x_b) \otimes x_c\bigr)=x_b\otimes x_ax_c-x_a\otimes x_bx_c\right\}_{(a,b,c)\in S}$ span $\ker f_{1,2}$.
\end{enumerate}
\begin{prop}\label{P:n1}
  For the balanced canonical ribbon $R$ of odd genus $g \geq 5$, we have $K_{1,2}(R) = 0$.
\end{prop}
\begin{proof}
For the proof, it suffices to exhibit a monomial basis of cosyzygies. We exhibit three such bases in Proposition~\ref{P:first-basis}, 
Proposition~\ref{P:second-basis}, and Proposition~\ref{P:third-basis}, respectively.
\end{proof}

Let $T \subset \GL(\HH^0(\omega))$ be the maximal torus acting
diagonally on the distinguished basis $\{x_i\}_{i=0}^{2k}$ of
$\HH^0(\omega)$. This basis yields a distinguished basis of
$\Gamma$ consisting of the $T$-eigenvectors $(x_a\wedge x_b)\otimes x_c$. 
Clearly, the monomial bases of cosyzygies correspond precisely to the non-zero
Pl\"{u}cker coordinates of $\CS_1(R)\in \Grass\bigl((3g-5)(g-1), \Gamma\bigr)$ with respect to this basis of
eigenvectors. To every such coordinate, and in turn, to every monomial
basis $\C$, we can associate a $T$-character, called the \emph{$T$-state of $\C$}. 
We may represent the $T$-state as a linear combination of $x_0, \dots, x_{2k}$. 
Precisely, the $T$-state of $\C=\{(x_a \wedge x_b)\otimes x_c\}_{(a,b,c)\in S}$ is given by
\[
w_T(\C):=\sum_{(a,b,c)\in S} w_T\bigl((x_a \wedge x_b)\otimes x_c\bigr) = \sum_{(a,b,c)\in S} (x_a+x_b+x_c)
=n_0x_0+\cdots +n_{2k} x_{2k},
\]
where $n_i$ is the number of occurrences of $x_i$ among the cosyzygies in $\C$. Note that we always have 
\[\sum_{i=0}^{2k} n_i=3(3g-5)(g-1).\]

We are now ready to prove our main theorem.
\begin{theorem}\label{T:ribbon-stability}
  Let $g \geq 5$ be odd. The balanced canonical ribbon of genus $g$ has $\SL_g$ semi-stable $1^{st}$ syzygy point.
\end{theorem}
\begin{proof}
Because $\HH^0(R,\omega)$ is a multiplicity-free representation of 
$\GG_m \subset \operatorname{Aut}(R)$ by Proposition \ref{P:mult-free}, it suffices to verify 
semi-stability of $\CS_1(R)$ with respect to the maximal torus
 $T$ acting diagonally on the basis $\{x_0,\dots, x_{2k}\}$ of
$\HH^0(R,\omega)$; see \cite[Proposition 4.7]{morrison-swinarski} and \cite[Proposition 2.4]{AFS-stability}. 

The non-zero Pl\"{u}cker coordinates of $\CS_1(R)$ diagonalizing the action of $T$ 
are precisely the monomial bases of cosyzygies. 
In Section \ref{S:monomial-bases}, we construct three monomial bases of cosyzygies,
$\C^+, \C^{-}$, and $\C^{\star}$ with the following $T$-states:
\begin{align*}
w_{T}(\mathcal C^+) &= (g^2-1)(x_0+x_k+x_{2k})+(6g-6)\sum_{i\neq 0,k,2k} x_i \, ,\\
w_{T}(\mathcal C^-) &= (7g-12)(x_0+x_{2k}) + (7g-15) x_k + (9g-18)\sum_{i \neq 0,k,2k} x_i \, ,\\
w_{T}(\mathcal C^\star) &= \frac{15g-29}{2}\left(x_0+x_{2k}\right) + (8g-16) x_k + (9g-20) \sum_{i \neq 0,k,2k} x_i \,.
\end{align*}
The $T$-semi-stability of $\CS_1(R)$ now follows from 
Lemma \ref{L:barycenter} below and the Hilbert-Mumford numerical criterion.
\end{proof}

\begin{corollary}[{Theorem \ref{T:main-stability}}]
  A general canonical curve of odd genus has a semi-stable $1^{st}$
  syzygy point.
\end{corollary}
\begin{proof}
This follows from the fact that $R$ deforms to a smooth canonical curve \cite{fong}.
\end{proof}

\begin{lemma}\label{L:barycenter} Suppose $g\geq 5$. Let $\C^+, \C^{-}$, and $\C^{\star}$ be the monomial
bases constructed in Section \ref{S:monomial-bases}.
Then the convex hull of the $T$-states 
$w_T(\mathcal C^+)$, $w_T(\mathcal C^-)$, and $w_T(\mathcal C^\star)$ contains the barycenter 
\[
\frac{3(3g-5)(g-1)}{g} (\sum_{i=0}^{2k} x_i).\]
\end{lemma}
\begin{proof}
Equivalently, we may show that the $0$-state is an effective linear combination of $w_T(\mathcal C^+)$, $w_T(\mathcal C^-)$, 
and $w_T(\mathcal C^\star)$ modulo $\sum_{i=0}^{2k} x_i$. 
We have
\begin{align*}
w_T(\mathcal C^+) &= (g-5)(g-1)(x_0 + x_k + x_{2k}) \pmod{ \sum_{i=0}^{2k} x_i }\, ,\\
w_T(\mathcal C^-) &= -(2g-6)(x_0 + x_{2k}) -(2g-3) x_k \pmod{ \sum_{i=0}^{2k} x_i }\, ,\\
w_T(\mathcal C^\star) &= - \frac{3g - 11}{2}(x_0+x_{2k}) - (g - 4) x_k \pmod{ \sum_{i=0}^{2k} x_i }\, .
\end{align*}
Form a positive linear combination $L$ of the last two lines as follows:
\begin{align*}
L &:= 6 w_T(\mathcal C^\star) + (g-3) w_T(\mathcal C^{-}) \\
&= -(2g^2-3g-15)(x_0 + x_k + x_{2k}) \pmod{ \sum_{i=0}^{2k} x_i }.
\end{align*}
Plainly, the $0$-state is a positive linear combination of $w_T(\mathcal C^+)$ and $L$.  
\end{proof}

\section{Construction of monomial bases of cosyzygies}
\label{S:monomial-bases}
 
In the remainder of this paper, we establish the existence of the three monomial bases of cosyzygies $\C^{+}, \C^{-}$, and $\C^{\star}$,
used in the proof of Theorem \ref{T:ribbon-stability}. 
This is done in Subsections \ref{first-basis}, \ref{second-basis}, and \ref{third-basis},
respectively.

\subsubsection*{Notation} Throughout this section, we use the following notation. 
We define the \emph{$u$-degree} of a cosyzygy $(x_a\wedge x_b)\otimes x_{c}$ 
to be $a+b+c$ and define the \emph{level} of a tensor $x_a\otimes x_bx_c$ 
to be $a$. By a slight abuse of notation, we often write $(x_a\wedge x_b)\otimes x_{c}$
to denote its image under $f_{2,1}$ in $\HH^0(\omega)\otimes \HH^0(\omega^2)$.

For $\alpha \in \QQ$, set $\{\alpha\} = \left\lfloor \alpha + \frac{1}{2} \right\rfloor$.
In other words, $\{\alpha\}$ is the integer closest to $\alpha$. 
Observe that for $n \in \ZZ$, we have
\[ n = \lfloor n/3 \rfloor + \{n/3\} + \lceil n/3 \rceil.\]
We use $\langle S\rangle$ to denote the linear span of elements in a subset $S$ of a vector space.

\subsubsection*{Outline of the construction}
We first describe our strategy for constructing monomial bases of cosyzygies. 
Recall from Definition \ref{D:monomial-basis} that a set $\C=\{(x_a \wedge x_b) \otimes x_c\}_{(a,b,c)\in S}\subset \Gamma$ of $(3g-5)(g-1)$ 
cosyzygies is a monomial basis of cosyzygies if and only if the images $f_{2,1}\bigl((x_a \wedge x_b) \otimes x_c\bigr)$, for 
$(a,b,c)\in S$, span $\ker f_{1,2}$.
The first step in our construction is to write down a set $\C$ of $(3g-5)(g-1)$ cosyzygies. We do this heuristically. 

Next, we make the following observation. Since $\operatorname{im} f_{2,1} \subseteq \ker f_{1,2}$ and 
$f_{1,2}$ is surjective onto $\HH^0(\omega^3)$,
to prove that the images of the cosyzygies in $\C$ span $\ker f_{1,2}$, it suffices to show that 
\[
\dim \left(\HH^0(\omega)\otimes \HH^0(\omega^2)\right) \big/ \langle f_{2,1}((x_a\wedge x_b)\otimes x_{c}) \rangle_{(a,b,c)\in S} \ \leq  \ 
\dim \HH^0(\omega^3)=5(g-1).
\]
In order to do this, we treat
\[
f_{2,1}\bigl((x_a\wedge x_b)\otimes x_{c}\bigr)=x_b\otimes x_a x_c- x_a \otimes x_b x_c
\]
as a relation among the elements of $\HH^0(\omega)\otimes \HH^0(\omega^2)$.
We therefore reduce to showing that the relations imposed by $\C$ reduce the dimension of
$\HH^0(\omega)\otimes \HH^0(\omega^2)$ to at most $5(g-1)$. 

The final observation is that all of our results and constructions are $\GG_m$-invariant.
In particular, we can run our argument degree by degree. This observation greatly simplifies our task 
because the relevant weight spaces have small dimensions. 
In particular, by Proposition \ref{thm:weight_decomposition_wm}
we have
\begin{equation}\label{E:cubic-weight-spaces}
\dim \HH^0(\omega^3)_d=
\begin{cases} 
1 & \text{if $0\leq d \leq k$ or $5k\leq d\leq 6k$,} \\ 
2 & \text{if $k < d < 5k$.}
\end{cases}
\end{equation}

\subsection{A construction of the first monomial basis}
\label{first-basis}
We define $\mathcal{C}^{+}$ to be the union of the following sets of cosyzygies:

\begin{enumerate}[label=\textbf{(T\arabic*)}] 
\item \label{a1} 

$(x_0 \wedge x_i)\otimes x_j$,        			
where $i\neq 0, 2k$ and $j\neq 2k$. 

\item \label{a2}

$(x_0 \wedge x_{i}) \otimes x_{2k}$, 		
where $1\leq i\leq k-1$. 

\item \label{a3}

$(x_0 \wedge x_{2k}) \otimes x_i$, 
where $i\leq k-1$. 


\item \label{a4}

$(x_{2k} \wedge x_{i})\otimes x_{j}$,  
where  $i\neq 0, 2k$ and $j\neq 0$. 

\item  \label{a5} 

$(x_{2k} \wedge x_{0}) \otimes x_i$, $i\geq k+1$. 

\item  \label{a6}

$(x_{2k} \wedge x_{i}) \otimes x_{0}$, $k+1\leq i\leq 2k-1$. 


\item  \label{a7}

 $(x_k \wedge x_i) \otimes x_j$, where $i\neq 0, k, 2k$ and $j\neq 0, 2k$. 

\item  \label{a8} 

$(x_k\wedge x_0) \otimes x_{2k}$ and $(x_k\wedge x_{2k}) \otimes x_0$. 

\item   \label{a9}

$(x_{i}\wedge x_{k+i}) \otimes x_{k-i}$ 
where $1\leq i \leq k-1$. 

\item   \label{a10} 

$(x_{2k-i}\wedge x_{k-i}) \otimes x_{k+i}$, $1\leq i \leq k-1$ 
\end{enumerate}


\begin{prop}\label{P:first-basis}
$\mathcal{C}^{+}$ is a monomial basis of cosyzygies with $T$-state 
\[
w_T(\mathcal{C}^+)=(g^2-1)(x_0+x_k+x_{2k})+(6g-6)\sum_{i\neq 0,k,2k} x_i.
\]
\end{prop}

\begin{proof}
Notice that $\mathcal{C}^+$ contains precisely $(3g-5)(g-1)$ cosyzygies and that it is invariant under the 
$\ZZ_2$-involution of the ribbon.

To calculate the $T$-state of $\mathcal{C}^+$, observe that $x_0$, $x_k$, $x_{2k}$ each appear $g^2-1$ times,
and $x_i$, for every $i\neq 0,k,2k$, appears $6g-6$ times.
It follows that 
\[
w_T(\mathcal{C}^+)=(g^2-1)(x_0+x_k+x_{2k})+(6g-6)\sum_{i\neq 0,k,2k} x_i.
\]

We now verify that $\mathcal{C}^+$ is a monomial basis of cosyzygies. In view of the $\ZZ_2$-symmetry
and the dimensions of $\HH^0(\omega^3)_d$ from \eqref{E:cubic-weight-spaces},
we only need to verify that the quotient space 
$\bigl(\HH^0(\omega)\otimes \HH^0(\omega^2)\bigr)_d/ \langle \mathcal{C}^{+}\rangle_{d}$ 
is at most one-dimensional in $u$-degrees $0\leq d\leq k-1$, and at most two-dimensional in $u$-degrees $k\leq d \leq 3k$.

The key player in our argument is the monomial basis $\B^+$ from Proposition \ref{P:quadratic-bases}:
\begin{equation}\label{E:B+}
\mathcal{B^{+}}=\left\{\{x_0x_i\}_{i=0}^{2k}, \ \{x_kx_i\}_{i=1}^{2k-1}, \ \{x_{2k}x_i\}_{i=1}^{2k}\right\}.
\end{equation}
Tensoring $\mathcal{B}^{+}$ with the standard basis $\{x_0, \dots, x_{2k}\}$ of $\HH^0(\omega)$,  
we obtain the following basis of $\HH^0(\omega)\otimes \HH^0(\omega^2)$:
\[
\B:=\{x_a \otimes \mathfrak{m} \co 0\leq a \leq 2k, \mathfrak{m}\in \mathcal{B}^{+}\}
\]

Our argument now proceeds by $u$-degree:

\subsubsection*{Degree $0\leq d\leq k$.} 
We have $\langle \B \rangle_d=\langle x_a\otimes x_0x_{d-a}\co  \text{where $0 \leq a\leq d$}\rangle$.
Evidently, we have $x_ax_{d-a}=x_0x_d$ in $\HH^0(\omega^2)$. It follows that  
\[
x_a\otimes x_0x_{d-a}=x_0\otimes x_ax_{d-a}+(x_0\wedge x_a)\otimes x_{d-a} = x_0\otimes x_0x_{d}  +(x_0\wedge x_a)\otimes x_{d-a},
\]
where $(x_0\wedge x_a)\otimes x_{d-a}$ is a cosyzygy \ref{a1}.
We conclude that $\langle \B\rangle_d/\langle\mathcal{C}^{+}\rangle_d$
is spanned by $x_0\otimes x_0x_{d}$, hence is at most one-dimensional.

\subsubsection*{Degree $k+1\leq d\leq 2k$.} 

We have \[
\langle \B  \rangle_d=\langle x_a\otimes x_0x_{d-a},\ x_b\otimes x_kx_{d-k-b}\co 0\leq a\leq d, \ 0\leq b< d-k\rangle.
\] 
If $b\geq 1$, using the cosyzygies \ref{a7} and \ref{a1} and Lemma \ref{L:quadratic}, we obtain
\begin{align*}
x_b\otimes x_kx_{d-k-b} &=x_k\otimes x_bx_{d-k-b}+(x_k\wedge x_b)\otimes x_{d-k-b} \\
&= x_k\otimes x_0x_{d-k}  + (x_k\wedge x_b)\otimes x_{d-k-b} \\
&=x_0\otimes x_k x_{d-k} + (x_0\wedge x_k)\otimes x_{d-k} + (x_k\wedge x_b)\otimes x_{d-k-b}.
\end{align*}
Using \ref{a1}, we also have  
\begin{align*}
x_a\otimes x_0x_{d-a}&=x_0\otimes x_ax_{d-a}+(x_0\wedge x_a)\otimes x_{d-a},
\end{align*}
It follows that 
$\langle \B \rangle_d/\langle\mathcal{C}^{+}\rangle_d=\langle  x_0\otimes x_ax_{d-a}\co 0\leq a\leq d \rangle/\langle\mathcal{C}^+\rangle_d$.
In other words, every tensor of $u$-degree $d$ is reduced to a tensor of level $0$.
Since
$\dim \langle  x_0\otimes x_ax_{d-a}\co 0\leq a\leq d \rangle =\dim \HH^0(\omega^2)_d=2$, we are done.


\subsubsection*{Degree $2k+1\leq d\leq 3k-1$.} Write $d=2k+i$, $1\leq i \leq k-1$. 
It is easy to see that modulo $\C^{+}$, every tensor in $\bigl(\HH^0(\omega)\otimes \HH^0(\omega^2)\bigr)_d$ 
can be reduced to a tensor of level $0$, $k$, or $2k$, by using cosyzygies \ref{a1}--\ref{a4}
or \ref{a7}. In other words,
\[
\langle \B \rangle_d /\langle \C^{+} \rangle_d 
=
\langle
x_0\otimes x_k x_{k+i}, \ x_0\otimes x_{2k} x_{i}, \ x_k\otimes x_{0}x_{k+i}, \ x_{k}\otimes x_{k}x_{i}, \ x_{2k}\otimes x_0x_{i} \rangle
/ \langle \C^+\rangle_d.
\]

Since $\dim \langle x_i\otimes x_ax_{2k-a}\co 0\leq a\leq 2k\rangle=\dim \HH^0(\omega^2)_{2k}=2$, it suffices to show that every tensor in the above display
can be rewritten modulo $\C^{+}$ as a tensor of level $i$. First, we observe that
\begin{align*}
x_{2k}\otimes x_0x_{i}  &=x_0 \otimes x_{2k}x_{i}+(x_{0}\wedge x_{2k}) \otimes x_{i} \qquad \ \text{(using \ref{a5} cosyzygy), }\\
x_0\otimes x_{i} x_{2k} &=x_i\otimes x_0x_{2k}-(x_0\wedge x_i)\otimes x_{2k} \qquad \ \text{(using \ref{a2} cosyzygy), } \\
x_{k}\otimes x_{i}x_{k} &=x_i \otimes x_k^2-(x_k\wedge x_i)\otimes x_k \qquad \qquad \text{(using \ref{a7} cosyzygy). } \\
\end{align*}
Since $x_k\otimes x_{0}x_{k+i}=x_0\otimes x_{k}x_{k+i}+(x_0\wedge x_{k})\otimes x_{k+i}$, it
remains to show that 
$x_0\otimes x_k x_{k+i}$ can be rewritten as a tensor of level $i$.  
To this end, we compute
\begin{align*}
x_0\otimes x_k x_{k+i} &=x_{k+i}\otimes x_0x_k -(x_0\wedge x_{k+i})\otimes x_{k}
=x_{k+i}\otimes x_ix_{k-i} -(x_0\wedge x_{k+i})\otimes x_{k}
\\
&=x_{i}\otimes x_{k+i}x_{k-i}+(x_i\wedge x_{k+i})\otimes x_{k-i}-(x_0\wedge x_{k+i})\otimes x_{k},
\end{align*}
where we have used a cosyzygy \ref{a9} in the second line.

\subsubsection*{Degree $d=3k$.} Using cosyzygies \ref{a1}, \ref{a4}, and \ref{a7},
every tensor in $\langle \B \rangle_{3k}$ reduces to a tensor of level $0$, $k$, or $2k$. 
It follows that 
\[
\langle \B \rangle_{3k} /\langle \C^{+} \rangle_{3k}
=
\langle
x_0\otimes x_k x_{2k}, \ 
x_{k}\otimes x_{0}x_{2k}, \ x_{k}\otimes x_{k}^2, \
x_{2k}\otimes x_{0}x_{k}  \rangle
/ \langle \C^+\rangle_{3k}.
\]
Using cosyzygies \ref{a8}, we see that $x_0\otimes x_k x_{2k}=x_{k} \otimes x_{0}x_{2k}$
and $x_{2k}\otimes x_{0}x_{k}=x_{k}\otimes x_{0}x_{2k}$ modulo $\C^{+}$. It follows
that $\langle \B \rangle_{3k}/ \langle \C^+\rangle_{3k}$ is spanned by tensors 
of level $k$, hence is at most two-dimensional.
\end{proof}

\subsection{A construction of the second monomial basis}
\label{second-basis}
We define $\mathcal{C}^{-}$ to be the union of the following sets of cosyzygies:
\begin{enumerate} [label=\textbf{(T\arabic*)}] 
\item\label{b_even}
  $(x_i \wedge x_j) \otimes x_j$, where $i \not \in \{j-k-1, j-k, j, j+k, j+k+1\}$.
\item\label{blow_odd}
 $(x_i \wedge x_{j+1}) \otimes x_j$, where $i > j+1$ or $i = j-k+1$, but $i \neq j+k$ and $i \neq j+k+1$.
\item\label{bhigh_odd}
  $(x_i \wedge x_{j-1}) \otimes x_j$, where $i < j-1$ or $i = j+k-1$, but $i \neq j-k$ and $i \neq j-k-1$.
\item \label{kblow_even}
 $(x_i \wedge x_j) \otimes x_{j+k}$, where $0 < j < k$ and $i \geq k$.
\item \label{kblow_odd}
  $(x_i \wedge x_j) \otimes x_{j+k+1}$, where $0 \leq j < k$ and $i \geq k$.
\item \label{kbhigh_even}
  $(x_i \wedge x_{j+k}) \otimes x_j$, where $0 < j < k$ and $i < k$.
\item \label{kbhigh_odd}
  $(x_i \wedge x_{j+k+1}) \otimes x_j$, where $0 \leq j < k$ and $i < k$.
\item \label{funny_k}
  $(x_k \wedge x_0) \otimes x_0$
\item \label{funny_5k}
  $(x_k \wedge x_{2k}) \otimes x_{2k}$
\item \label{funny}
  $(x_{\lfloor (d-2k)/3 \rfloor} \wedge x_{\lceil (d+2k)/3 \rceil}) \otimes x_{d- \lfloor (d-2k)/3 \rfloor - \lceil (d+2k)/3 \rceil}$, 
  where $2k \leq d \leq 4k$ with the following exception: If $k \equiv 1 \pmod 3$ and $d = 2k$, 
  then take instead \\ $(x_0 \wedge x_{\lfloor 4k/3 \rfloor}) \otimes x_{\lceil 2k/3 \rceil}$.
\end{enumerate}

  The construction of $\mathcal{C}^-$ is motivated by the following basis of $\HH^0(\omega^2)$ from Proposition \ref{P:quadratic-bases}:
  \begin{equation*}
     \mathcal{B}^{-} = \left\{ \{x_j^2\}_{j=0}^{2k}, \{x_jx_{j+1}\}_{j=0}^{2k-1}, \{x_jx_{j+k}\}_{j=1}^{k-1}, \{x_jx_{j+k+1}\}_{j=0}^{k-1}\right\}.
   \end{equation*}
   After tensoring with $\{x_0, \dots, x_{2k}\}$, the basis above yields  the  basis of $\HH^0(\omega) \otimes \HH^0(\omega^2)$ given by
   \begin{equation*}
     \B := \{ x_i \otimes \mathfrak{m} \mid 0 \leq i \leq 2k,\ \mathfrak{m} \in \mathcal{B}^-\}.
   \end{equation*}


\begin{prop}\label{P:second-basis}
  $\mathcal{C}^{-}$ is a monomial basis of cosyzygies with $T$-state
  \[
    w_T(\mathcal C^-) = (7g-12)(x_0+x_{2k}) + (7g-15) x_k + (9g-18)
    \sum_{i \neq 0,k,2k} x_i.
\]
\end{prop}
\begin{remark}
  The exception for $k \equiv 1 \pmod 3$ and $d = 2k$ in \ref{funny} is only to get the correct $T$-state for $\mathcal C^-$. 
  One obtains a monomial basis regardless.
\end{remark}
\begin{proof}
  Counting cosyzygies of each type in $\mathcal C^-$,
  we get $12k^2-4k = (3g-5)(g-1)$ cosyzygies. The state calculation is also straightforward.

Let $\Lambda$ be the span 
in $\HH^0(\omega) \otimes  \HH^0(\omega^2)$ of all cosyzygies in $\C^{-}$
and let $\Lambda'$ be the span 
in $\HH^0(\omega) \otimes \HH^0(\omega^2)$ of the cosyzygies \ref{b_even}--\ref{kbhigh_odd}.
The relations given by $\Lambda$ reduce a tensor in $\B$ to a different tensor. For example, modulo \ref{b_even} we have
\[ x_i \otimes x_j^2  = x_j \otimes x_ix_j.\]
Our goal is to show that the quotient $\left(\HH^0(\omega) \otimes  \HH^0(\omega^2)\right)/\Lambda$
is generated by at most one element in degrees $0 \leq d \leq k$ and $5k \leq d \leq 6k$, and by at most two elements in degrees $k < d < 5k$. 
Proposition~\ref{P:markov_chain} does most of the heavy lifting towards this goal and, for the sake of the argument,
we assume its statement for now. 

By Proposition~\ref{P:markov_chain} and Remark~\ref{rem:sinks_in_d}, $\left(\HH^0(\omega) \otimes \HH^0(\omega^2)\right)/\Lambda'$ 
is generated by one element in degrees $0 \leq d < k$ and $5k < d \leq 6k$, by two elements 
in degrees $k \leq d < 2k$ and $4k < d \leq 5k$ and by three elements in degrees $2k \leq d \leq 4k$. 
Therefore to complete the argument, it suffices to prove that the cosyzygies \ref{funny_k} and \ref{funny_5k}
impose nontrivial linear relations on the two generators in degree $k$ and $5k$, respectively,
and that the cosyzygy \ref{funny} imposes a nontrivial linear relation on the three generators in degrees $2k \leq d \leq 4k$.

Let $d = k$. The two generators of $\left(\HH^0(\omega) \otimes \HH^0(\omega^2)\right)/\Lambda'$ in this degree are
\begin{align*}
\sigma_1 &:= x_{\{k/3\}} \otimes x_{\lfloor k/3 \rfloor}x_{\lceil k/3 \rceil}, \text{ and }\\
    \sigma_2 &:= x_k \otimes x_0^2.
  \end{align*}
  The relation imposed by \ref{funny_k} is
  \[ x_k \otimes x_0^2 = x_0 \otimes x_0 x_k.\]
  It is easy to see that modulo \ref{b_even}, \ref{blow_odd}, and \ref{bhigh_odd}, we have
  \[ x_0 \otimes x_0 x_k = x_0 \otimes x_{\lfloor k/2 \rfloor}x_{\lceil k/2 \rceil} = \sigma_1.\]
  Therefore, \ref{funny_k} imposes the nontrivial relation
  \[ \sigma_2 = \sigma_1.\]

  The case of $d = 5k$ follows symmetrically.
  
  Let $2k \leq d \leq 4k$. The three generators of $\left(\HH^0(\omega) \otimes \HH^0(\omega^2)\right)/\Lambda'$ in degree $d$ are
  \begin{align*}
    \sigma_1 &:= x_{\{d/3\}} \otimes x_{\lceil d/3 \rceil}x_{\lfloor d/3 \rfloor},\\
    \sigma_2 &:= x_{\lceil (d+2k)/3 \rceil} \otimes x_{\lfloor (d-k)/3 \rfloor} x_{\{ (d-k)/3\}}, \text{ and }\\
    \sigma_3 &:= x_{\lfloor (d-2k)/3 \rfloor} \otimes x_{\{ (d+k)/3 \}} x_{\lceil (d+k)/3\rceil}.
  \end{align*}
  For brevity, set $\ell = \lceil (d+2k)/3 \rceil$ and $s = \lfloor (d-2k)/3 \rfloor$. The relation imposed by \ref{funny} is
  \begin{equation} \label{eqn:funny}
    x_\ell \otimes x_s x_{d-\ell-s} = x_s \otimes x_\ell x_{d-\ell-s}.
  \end{equation}
  Assume that $d \leq 3k$; the case of $d \geq 3k$ follows symmetrically. Since $d \leq 3k$, we have
  \[ s < \lfloor (d-k)/3\rfloor \leq \{(d-k)/3 \} < d-\ell-s \leq k.\]
  On the left hand side of \eqref{eqn:funny}, we have by Lemma \ref{L:quadratic}
  \begin{align*}
    x_\ell \otimes x_s x_{d-\ell-s} &= x_\ell \otimes x_{\lfloor (d-k)/3 \rfloor}x_{\{(d-k)/3\}}\\
    &= \sigma_2.
  \end{align*}
    On the right hand side of \eqref{eqn:funny}, working modulo \ref{kbhigh_even}--\ref{kbhigh_odd}, and applying Lemma \ref{L:quadratic}, we get
  \begin{align*}
    x_s \otimes x_\ell x_{d-\ell-s} &= \lambda x_{s} \otimes x_{\{(d+k)/3\}}x_{\lceil(d+k)/3\rceil} + \mu x_s \otimes x_{\lfloor (d-s-k)/2 \rfloor} x_{\lceil (d-s+k)/2 \rceil}\\
    &= \lambda \sigma_3 + \mu x_{\lceil (d-s+k)/2 \rceil} \otimes m, \text{ where $m$ is balanced},\\
    &= \lambda \sigma_3 + \mu(\alpha \sigma_1 + \beta \sigma_2 + \gamma \sigma_3),
  \end{align*}
  where the last step uses Proposition~\ref{P:markov_chain}. Furthermore,
  since $\lceil (d-s+k)/2 \rceil > \lceil (d+2k)/3 \rceil$, Proposition~\ref{P:markov_chain} (Part 2(c)) implies that $\alpha < 0$. Thus, \ref{funny} imposes
  the relation
  \[ \sigma_2 = \lambda \sigma_3 + \mu(\alpha \sigma_1 + \beta \sigma_2 + \gamma \sigma_3).\]
  If $\mu = 0$, then this relation is clearly nontrivial. If $\mu \neq 0$, the non-vanishing of the coefficient of $\sigma_1$ shows that the relation is nontrivial.

  Finally, we verify that the exceptional cosyzygy in \ref{funny} for $k \equiv 1 \pmod 3$ and $d = 2k$ imposes a nontrivial relation. The argument is almost the same. In this case, the cosyzygy gives
  \begin{equation}\label{eqn:exc_funny}
  x_{\lfloor 4k/3 \rfloor} \otimes x_0x_{\lceil 2k/3 \rceil}=x_0 \otimes x_{\lceil 2k/3 \rceil}x_{\lfloor 4k/3 \rfloor}.
  \end{equation}

  Reducing the left hand side of \eqref{eqn:exc_funny} modulo \ref{b_even}--\ref{kbhigh_odd}, we get
  \begin{align*}
    x_{\lfloor 4k/3 \rfloor} \otimes x_0x_{\lceil 2k/3 \rceil} &= x_{\lfloor 4k/3 \rfloor} \otimes m, \text{ where $m$ is balanced,}\\
    &= \alpha \sigma_1 + \beta\sigma_2 + \gamma \sigma_3.
  \end{align*}
  Since $\lfloor (d-2k)/3 \rfloor \leq \lfloor 4k/3 \rfloor \leq \lceil (d+2k)/3 \rceil$, Proposition~\ref{P:markov_chain} (Part 2(b)) implies that $\alpha > 0$.  

  Reducing the right hand side of \eqref{eqn:exc_funny}, we get
  \begin{align*}
    x_0 \otimes x_{\lceil 2k/3 \rceil}x_{\lfloor 4k/3 \rfloor} &= \lambda x_0 \otimes x_k^2 + \mu x_0 \otimes x_{\lfloor k/2 \rfloor} x_{\lceil 3k/2 \rceil}, \text{ where $\lambda, \mu > 0$} \\
    &= \lambda \sigma_3 + \mu x_{\lceil 3k/2\rceil} \otimes m,\text{ where $m$ is balanced,}\\
    &= \lambda \sigma_3 + \mu (\alpha' \sigma_1 + \beta'\sigma_2 + \gamma' \sigma_3).
  \end{align*}
  Since $\lceil 3k/2 \rceil > \lceil (d+2k)/3 \rceil$, Proposition~\ref{P:markov_chain} (Part 2(c)) implies that $\alpha' < 0$. Thus, \ref{funny} imposes
  \[ \alpha \sigma_1 + \beta\sigma_2 + \gamma \sigma_3 = \lambda \sigma_3 + \mu (\alpha' \sigma_1 + \beta'\sigma_2 + \gamma' \sigma_3).\]
  Since $\alpha > 0$ whereas $\mu\alpha' < 0$, the relation is nontrivial. 
\end{proof}

Before moving onto the key technical results needed in the proof of Proposition \ref{P:second-basis}, we introduce
some additional terminology. We call the forms $x_j^2$ and $x_jx_{j+1}$ \emph{balanced} and the forms $x_jx_{j+k}$ and $x_jx_{j+k+1}$ \emph{$k$-balanced}. Likewise, we 
call a tensor $x_i \otimes \mathfrak{m}$ \emph{balanced} (resp. \emph{$k$-balanced}) if $\mathfrak{m}$ is \emph{balanced} (resp. \emph{$k$-balanced}). Finally, we call a 
balanced tensor $x_i \otimes \mathfrak{m}$ of degree $d$
 \emph{well-balanced} if $\lfloor (d-2k)/3 \rfloor \leq i \leq \lceil (d+2k)/3 \rceil$. 
 Equivalently, a balanced tensor $x_i \otimes x_s x_\ell$ is well-balanced if $\max(|i-s|, |i-\ell |) \leq k+1$.
\begin{prop}

\label{P:markov_chain}
(Part 1) 
Every element of $\left(\HH^0(\omega) \otimes \HH^0(\omega^2)\right)/\Lambda'$ can be uniquely expressed as a linear combination of the 
following tensors:
  \begin{align*}
    &\text{Type 1}\left \{
    \begin{array}{l l}
    x_i \otimes x^2_i, \quad\text{ where } 0 \leq i \leq 2k\\
    x_i \otimes x_i x_{i+1}, \quad\text{ where } 0 \leq i \leq 2k-1\\
    x_i \otimes x_{i-1}x_i, \quad\text{ where } 1 \leq i \leq 2k\\
  \end{array}\right.\\  
&\text{Type 2}\left \{
  \begin{array}{l l}
  x_{i+k} \otimes x_i^2, \quad\text{ where } 0 \leq i \leq k\\
    x_{i+k+1} \otimes x_i^2, \quad\text{ where } 0 \leq i \leq k-1\\
    x_{i+k+1} \otimes x_ix_{i+1}, \quad\text{ where } 0 \leq i \leq k-1\\
  \end{array}\right.\\  
&\text{Type 3}\left \{
 \begin{array}{l l}
    x_{i-k} \otimes x_i^2, \quad\text{ where } k \leq i \leq 2k\\
    x_{i-k-1} \otimes x_i^2, \quad\text{ where } k+1 \leq i \leq 2k\\
    x_{i-k-1} \otimes x_ix_{i-1}, \quad\text{ where } k+1 \leq i \leq 2k\\
  \end{array}\right.
\end{align*}

(Part 2) Furthermore, let $2k \leq d \leq 4k$. Then there is precisely one tensor of degree $d$ of each Type 1--3. Suppose the balanced tensor
$\tau = x_i \otimes x_{\lfloor (d-i)/2\rfloor} x_{\lceil (d-i)/2
  \rceil}$ is expressed as
\[ \tau = \alpha \sigma_1 + \beta \sigma_2 + \gamma \sigma_3,\]
where $\sigma_t$ is of Type $t$. Then,
\begin{enumerate}
\item[(a)] $\alpha + \beta + \gamma = 1$;
\item[(b)] if $\lfloor (d-2k)/3 \rfloor < i < \lceil (d+2k)/3 \rceil$, then $\alpha > 0$, $\beta \geq 0$, $\gamma \geq 0$ (well-balanced case);
\item[(c)] if $i > \lceil (d+2k)/3 \rceil$, then $\alpha < 0$, $\beta > 1$, $\gamma \leq 0$;
\item[(d)] if $i < \lfloor (d-2k)/3 \rfloor$, then $\alpha < 0$, $\beta \leq 0$, $\gamma > 1$.
\end{enumerate}
\end{prop}

\begin{remark}\label{rem:sinks_in_d}
  In terms of the $u$-degree, the list of tensors in
  Proposition~\ref{P:markov_chain} can be written more compactly as follows:
\begin{enumerate}
 \item[](Type 1) $x_{\{d/3\}} \otimes x_{\lfloor d/3 \rfloor} x_{\lceil d/3 \rceil}$ \quad where $0 \leq d \leq 6k$.
 \item[](Type 2) $x_{\lceil (d+2k)/3 \rceil} \otimes x_{\lfloor (d-k)/3 \rfloor} x_{\{(d-k)/3 \}}$ \ where $k \leq d \leq 4k$.
 \item[](Type 3) $x_{\lfloor (d-2k)/3 \rfloor} \otimes x_{\lceil (d+k)/3 \rceil} x_{\{(d+k)/3 \}}$ \ where $2k \leq d \leq 5k$.
\end{enumerate}
\end{remark}

\begin{proof}
Using the cosyzygies \ref{b_even}--\ref{kbhigh_odd}, we reduce every element of the basis $\B$ 
to a linear combination of the tensors of Type 1, 2, and 3. 
Uniqueness then follows by counting the dimensions. 
  
\textbf{Step 1} (Reducing $k$-balanced tensors to balanced tensors):
Consider a $k$-balanced tensor $x_i \otimes x_{\lfloor (d-i-k)/2 \rfloor} x_{\lceil(d-i+k)/2\rceil}$, where $k \leq d-i \leq 3k$. Suppose $i \geq k$. 
Then modulo the cosyzygy \ref{kblow_even} or \ref{kblow_odd}, we get
\[ x_i \otimes x_{\lfloor (d-i-k)/2 \rfloor} x_{\lceil(d-i+k)/2\rceil} = x_{\lfloor (d-i-k)/2 \rfloor} \otimes x_i x_{\lceil(d-i+k)/2\rceil}.\]
Since $i \geq k$ and $\lceil(d-i+k)/2\rceil \geq k$, the form $x_ix_{\lceil (d-i+k)/2 \rceil}$ equals a balanced form in $\HH^0(\omega^2)$
by Lemma \ref{L:quadratic}. The case of $i < k$ is analogous using cosyzygies \ref{kbhigh_even} or \ref{kbhigh_odd}.

\textbf{Step 2} (Reducing balanced tensors to well-balanced tensors):
Consider a balanced tensor $x_i \otimes x_{\lfloor (d-i)/2\rfloor}x_{\lceil (d-i)/2 \rceil}$ that is not well-balanced. For brevity, set
  \[ s = \lfloor(d-i)/2 \rfloor, \quad \ell = \lceil(d-i)/2\rceil.\]
  Assume that $i > \lceil (d+2k)/3 \rceil$ (the case of $i < \lfloor(d-2k)/3 \rfloor$ follows symmetrically). We then have 
  $i - s > k+1$ and hence $i > \ell + k > \ell$. Modulo the cosyzygy \ref{b_even} or \ref{blow_odd}, we get
  \[  x_i \otimes x_s x_\ell =  x_\ell \otimes x_s x_i.\]
  By Lemma \ref{L:quadratic}, we have
  \[ x_sx_i = \lambda m_1 + \mu m_1',\]
  where $m_1$ is balanced, $m_1'$ is $k$-balanced, and $\lambda + \mu = 1$. Since $i - s > k+1$, we also have $\lambda < 0$. 
  Reducing the $k$-balanced tensor $x_\ell \otimes m_1'$ as in Step 1, we get
  \begin{align*}
    x_\ell\otimes m_1' &= x_\ell \otimes x_{\lfloor (d - \ell -k)/2\rfloor} x_{\lceil (d - \ell +k)/2 \rceil} \\
    &= x_{\lceil (d-\ell+k)/2 \rceil} \otimes m_2 \quad \text{ modulo \ref{kbhigh_even} or \ref{kbhigh_odd}}
  \end{align*}
  where $m_2$ is balanced. We thus get an expression
  \begin{equation}\label{eqn:deviation_reduction}
      x_i \otimes x_sx_\ell = \lambda x_\ell \otimes m_1 + \mu x_{\lceil (d-\ell+k)/2 \rceil} \otimes m_2, 
     \end{equation}
\[
\text{where $m_1$ and $m_2$ are balanced, } s = \lfloor (d-i)/2 \rfloor,\ \ell = \lceil(d-i)/2\rceil,\ \lambda + \mu = 1,\ \lambda < 0.
\]

  Note that we have the inequalities
  \begin{align*}
    &\lfloor (d-2k)/3 \rfloor \leq \ell \leq \lceil (d+2k)/3 \rceil, \text { and } \\
    &\lceil (d+2k)/3 \rceil \leq \lceil (d-\ell+k)/2\rceil < i.
  \end{align*}
  In other words, the first tensor on the right in \eqref{eqn:deviation_reduction} is already well-balanced and the second is strictly closer to being well-balanced than the original tensor. By repeated application of \eqref{eqn:deviation_reduction}, we arrive at a linear combination of well-balanced tensors.
  
  \textbf{Step 3} (Reducing the well-balanced tensors): We now show that all well-balanced tensors reduce to linear combinations of tensors of Type 1, 2, and 3. We will make use of the following result.
  \begin{lemma}\label{L:markov}
    Let $\tau=x_i\otimes x_{\lfloor (d-i)/2 \rfloor}x_{\lceil (d-i)/2 \rceil}$ be a well-balanced tensor of degree $d$. Modulo
    \emph{\ref{b_even}--\ref{kbhigh_odd}}, we have a reduction
    \begin{equation}\label{eqn:markov_reduction}
      \tau = \lambda \tau_1 + \mu \tau_2,
    \end{equation}
    where $\tau_1$ and $\tau_2$ are well-balanced, $\lambda + \mu = 1$, and $\lambda,\ \mu \geq 0$. Moreover, if $\tau$ is not of Type 2 or 3, then $\lambda >0$. And, if $\tau$ is not of Type 1, 2, or 3, then $\tau_1=x_j \otimes x_{\lfloor (d-j)/2 \rfloor}x_{\lceil (d-j)/2 \rceil}$,
    where $\vert \{d/3\} - j \vert < \vert \{d/3\} - i \vert$, and $\tau_1$ is not of Type 2 or 3.
  \end{lemma}
  \begin{proof}[Proof of the lemma]
    Let $\tau = x_i \otimes x_{\lfloor (d-i)/2 \rfloor}x_{\lceil (d-i)/2 \rceil}$. For brevity, set $s= \lfloor (d-i)/2 \rfloor$ and $\ell = \lceil (d-i)/2 \rceil$. 
    If $i = \ell$ or $i = s$, then $\tau$ is of Type 1. In this case, we take $\tau_1 = \tau$ and $\lambda = 1,\ \mu = 0$. If both $x_i x_{\ell}$ and $x_i x_{s}$ are $k$-balanced, 
    then $\tau$ is of Type 2 or 3. In this case, we take $\tau_2 = \tau$ and $\lambda = 0, \mu = 1$. Suppose neither of these is the case. 
    We consider the case of $i > \ell$; the case of $i < s$ follows symmetrically. Note that $\ell$ satisfies
    \begin{equation}\label{eqn:j_close}
      \{ (d-k)/3 \} \leq \ell \leq \{ (d+k)/3 \}.
    \end{equation}

    We first treat the special case $i = s + k$. Since not both $x_i x_\ell$ and $x_ix_s$ are $k$-balanced, we must have $s = \ell-1$. 
    Therefore, we get
    \begin{align*}
      \tau &=  x_i \otimes x_{\ell-1} x_\ell \\
      &= x_{\ell-1} \otimes x_i x_\ell \quad \text{ modulo \ref{bhigh_odd}}\\
      &= \lambda x_{\ell-1} \otimes m_1 + \mu x_{\ell-1} \otimes x_{\ell+k} x_{\ell-1}, \\
      &\qquad \text{ where $m_1$ is balanced, $\lambda > 0$, $\mu \geq 0$, and $\lambda + \mu = 1$ (Lemma \ref{L:quadratic}),}\\
      &= \lambda x_{\ell-1} \otimes m_1 + \mu x_{\ell+k} \otimes x_{\ell-1}x_{\ell-1} \quad \text{ modulo \ref{kbhigh_even}}\\
      &= \lambda \tau_1 + \mu \tau_2, \quad \text{as desired.}
    \end{align*}

    Now assume that $i \neq s+k$. Then $0 < i-\ell \leq i-s < k$. In this case, we get
    \begin{align*}
      \tau 
      &=  x_i \otimes x_s x_\ell \\
      &= x_\ell \otimes x_s x_i \quad \text{ modulo \ref{b_even} or \ref{blow_odd}}.
    \end{align*}
    We now write using Lemma \ref{L:quadratic}
    \[ x_s x_i = \lambda m_1 + \mu m_1',\]
    where $m_1$ is balanced, $m_1'$ is $k$-balanced and $\lambda + \mu = 1$. Since $0 < i- s < k$, we have $\lambda > 0$ and $\mu \geq 0$. Reducing the $k$-balanced tensor $x_\ell \otimes m_1'$ as in Step 1, we get
    \[ x_\ell \otimes m_1' = x_p \otimes m_2,\]
    where $m_2$ is balanced and 
    \[ p = 
    \begin{cases}
      \lfloor (d-\ell-k)/2 \rfloor & \text{if $\ell \geq k$,}\\
      \lceil (d-\ell+k)/2 \rceil & \text{if $\ell < k$.}
    \end{cases}
    \]
    In either case, \eqref{eqn:j_close} implies that
    \[ \lfloor (d-2k)/3 \rfloor \leq p \leq \lceil (d+2k)/3 \rceil.\]
    Setting $\tau_1 = x_\ell \otimes m_1$ and $\tau_2 = x_p \otimes m_2$, we thus get
    \[ \tau = \lambda \tau_1 + \mu \tau_2,\]
    as claimed.
    
    Finally, we note that if $\tau=x_i \otimes x_{s}x_{\ell}$ was not of Type 1, 2, or 3, then by construction
    $\tau_1$ has level $j$ where either $j=s$
    in the case of $i = s + k$, or $j=\ell$ in all other cases. In either case, it is clear
    that $\vert \{d/3\} - j \vert < \vert \{d/3\} - i \vert$. (Informally, this means that $\tau_1$ is closer to being Type 1 than 
    $\tau$.) This finishes the proof of the lemma.
  \end{proof}
  
  \begin{figure}
    \begin{tikzpicture}[node distance=5em]
      \node[draw,rectangle] (0) {$\tau$};
      \node[draw,rectangle, below left of=0] (1) {$\tau_1$};
      \node[draw,rectangle, below right of=0] (2) {$\tau_2$};
      \draw (0) edge [->, bend right=20] node [left] {$\lambda$} (1);
      \draw (0) edge [->, bend left=20] node [right] {$\mu$} (2);
    \end{tikzpicture}
    \caption{The relations among well-balanced tensors as a Markov chain}
    \label{fig:markov}
  \end{figure}
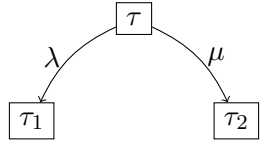
  
  We continue the proof of Proposition~\ref{P:markov_chain}. Let $\Omega$ be the set of well-balanced tensors. 
  Define a linear operator $P \colon \CC\langle\Omega\rangle \to \CC\langle\Omega\rangle$ that encodes \eqref{eqn:markov_reduction}, namely
  \[ P \colon \tau \mapsto \lambda \tau_1 + \mu \tau_2.\]
  By Lemma \ref{L:markov}, we can interpret $P$ as a Markov process on $\Omega$ (see Figure~\ref{fig:markov}). 
  Notice that the absorbing states of this Markov chain are precisely the tensors of 
  Type 1, 2, and 3. 
  Furthermore, from every other tensor, the path $\tau \to \tau_1 \to \dots$ eventually leads to a tensor of Type 1, again 
  by Lemma \ref{L:markov}. As a result, $P$ is an absorbing Markov 
  chain. By basic theory of Markov chains, 
  for every $v \in \CC\langle\Omega\rangle$, the limit $\lim_{n \to \infty} P^nv$ exists and is supported on the absorbing states. 
  Taking $v = 1 \cdot \tau$, we conclude 
  that $\tau$ reduces to a linear combination of the absorbing states. We thus get a linear relation
  \[ \tau = \alpha \sigma_1 + \beta \sigma_2 + \gamma \sigma_3,\]
  where $\sigma_t$ is of Type $t$, as claimed.
  
  The above analysis also lets us deduce the claims about the coefficients from Part 2 of the proposition. 
  Let $2k \leq d \leq 4k$. Say $\tau = x_i \otimes x_s x_\ell$ reduces as
  \[ \tau = \alpha \sigma_1 + \beta \sigma_2 + \gamma \sigma_3,\]
  where $\sigma_t$ is of Type $t$. 

  For Part 2(a), we note that $\alpha+\beta+\gamma = 1$ follows by passing to $\HH^0(\omega^3)$ and comparing the coefficients
  of $u^d$.

  For Part 2(b), assume that $\lfloor (d-2k)/3 \rfloor < i < \lceil (d+2k)/3 \rceil$. Then $\tau$ is well-balanced. The non-negativity of $P$ implies the non-negativity of 
  $\alpha$, $\beta$, and $\gamma$. Furthermore, since there is a path of positive weight from $\tau$ to $\sigma_1$, we have $\alpha > 0$.

  For Part 2(c), note that if $i = \lceil (d+2k)/3 \rceil$, then $\alpha = 0$, $\beta = 1$, and $\gamma = 0$. For $i > \lceil (d+2k)/3\rceil$, we show by descending induction on $i$ that $\alpha < 0$ and $\gamma \leq 0$. Then since $\alpha+\beta+\gamma = 1$, it follows that $\beta > 1$. 
  For the induction, recall the reduction \eqref{eqn:deviation_reduction}:
  \begin{equation*}\label{eqn:deviation_reduction_again}
  \tau = \lambda x_\ell \otimes m_1 + \mu x_{\lceil (d-\ell+k)/2 \rceil} \otimes m_2,
\end{equation*}
where the $m_i$ are balanced, $\lambda < 0$, $\mu > 0$, and $\lambda + \mu = 1$. Recall also the inequalities
  \begin{align*}
    \lfloor (d-2k)/3 \rfloor &\leq \ell \leq \lceil (d+2k)/3 \rceil \text{ and }\\
    \lceil (d+2k)/3 \rceil &\leq \lceil (d-\ell+k)/2 \rceil < i.
  \end{align*}
  Except in the extreme case $(d,i) = (2k,2k)$, both inequalities in the first line are strict. Say we have the reductions
  \begin{align*}
    x_\ell \otimes m_1 &= \alpha' \sigma_1 + \beta' \sigma_2 + \gamma' \sigma_3 \text{, and }\\
    x_{\lceil (d-\ell+k)/2 \rceil} \otimes m_2 &=\alpha'' \sigma_1 + \beta'' \sigma_2 + \gamma'' \sigma_3.
  \end{align*}
By Part 2(b), we have $\alpha' > 0$, and $\gamma' \geq 0$. By the inductive assumption, we have $\alpha'' \leq 0$, and $\gamma'' \leq 0$. 
Since $\lambda < 0$ and $\mu > 0$ in \eqref{eqn:deviation_reduction_again}, we conclude the induction step. In the extreme case $(d,i) = (2k,2k)$, 
the reduction \eqref{eqn:deviation_reduction_again} becomes
  \[ \tau = \lambda \sigma_3 + \mu x_{\lceil 3k/2 \rceil} \otimes m_2.\]
  The assertion now follows from that for $x_{\lceil 3k/2 \rceil} \otimes m_2$.

  Finally, Part 2(d) follows symmetrically from Part 2(c).

\end{proof}

\subsection{A construction of the third (and final!) monomial basis}
\label{third-basis}
Let $\mathcal C^\star$ be the union of the following sets of cosyzygies:
\begin{enumerate}[label=\textbf{(S\arabic*)}] 
\item\label{old} The cosyzygies \ref{b_even}--\ref{funny_5k} in the description of $\mathcal{C}^{-}$.
\item \label{fudge_low} $(x_{d-k} \wedge x_0) \otimes x_k$ for $2k \leq d < 3k$,
\item \label{fudge_middle} $(x_{2k} \wedge x_0) \otimes x_k$
\item \label{fudge_high} $(x_{d-3k} \wedge x_{2k}) \otimes x_k$ for $3k < d \leq 4k$.
\end{enumerate}

\begin{prop}\label{P:third-basis}
  $\mathcal{C}^\star$ is a monomial basis of cosyzygies with $T$-state
  \[
  w_T(\mathcal C^\star) =
  \frac{15g-29}{2}\left(x_0+x_{2k}\right) + (8g-16) x_k + (9g-20) \sum_{i \neq 0,k,2k} x_i
  \]
\end{prop}
\begin{proof}
  Let $\Lambda'$ be the span in $\HH^0(\omega) \otimes \HH^0(\omega^2)$ of the cosyzygies in \ref{old}. Then by Proposition \ref{P:markov_chain}, 
  the quotient 
  $\left(\HH^0(\omega) \otimes \HH^0(\omega^2)\right)/\Lambda'$ is generated by one element in degrees $0 \leq d \leq k$ and $5k \leq d \leq 6k$, by two elements in degrees 
  $k < d < 2k$ and $4k < d < 5k$, and by three elements in degrees $2k \leq d \leq 4k$. It suffices to prove that the cosyzygies \ref{fudge_low}--\ref{fudge_high} impose a 
  nontrivial linear relation among the three generators in degrees $2k \leq d \leq 4k$.

Let $2k \leq d < 3k$. Recall that the three generators in this
degree are
\begin{align*}
\sigma_1 &:= x_{\{d/3\}} \otimes x_{\lceil d/3 \rceil}x_{\lfloor d/3 \rfloor},\\
    \sigma_2 &:= x_{\lceil (d+2k)/3 \rceil} \otimes x_{\lfloor (d-k)/3 \rfloor} x_{\{ (d-k)/3\}}, \text{ and }\\
    \sigma_3 &:= x_{\lfloor (d-2k)/3 \rfloor} \otimes x_{\{ (d+k)/3 \}} x_{\lceil (d+k)/3\rceil}.\\
  \end{align*}
  The relation given by \ref{fudge_low} is 
  \[ x_0 \otimes x_{d-k}x_{k} = x_{d-k} \otimes x_0 x_k.\]
  We reduce both sides modulo $\Lambda'$. Note that $x_0 \otimes x_{d-k}x_k = x_0 \otimes m_1$ and $x_{d-k} \otimes x_0x_k = x_{d-k} \otimes m_2$ where the $m_i$ are balanced. Modulo $\Lambda'$, we have by Proposition \ref{P:markov_chain}
  \begin{align*}
    x_0 \otimes m_1 &= \alpha \sigma_1 + \beta \sigma_2 + \gamma \sigma_3 \ ,\quad \text{and}\\
    x_{d-k} \otimes m_2 &= \alpha' \sigma_1 + \beta' \sigma_2 + \gamma' \sigma_3.
  \end{align*}
  The relation imposed by \ref{fudge_low} is therefore
  \begin{equation}\label{eqn:fudge_relation_low}
 \alpha \sigma_1 + \beta \sigma_2 + \gamma \sigma_3 = \alpha' \sigma_1 + \beta' \sigma_2 + \gamma' \sigma_3.
\end{equation}
On one hand, since $0 \leq \lfloor (d-2k)/3 \rfloor$, Proposition~\ref{P:markov_chain} implies that $\gamma > 0$, $\alpha \leq 0$, and $\beta \leq 0$. On the other hand, since $\lfloor (d-2k)/3 \rfloor < d-k$, we either have $\alpha > 0$ (if $d-k < \lceil (d+2k)/3 \rceil$) or $\beta > 0$ (if $\lceil (d+2k)/3 \rceil \leq d-k$). In either case, the relation \eqref{eqn:fudge_relation_low} is nontrivial.

The same argument goes through for $d = 3k$.

The case of $3k < d \leq 4k$ follows symmetrically.
\end{proof}

\section{Computer calculations}  
For any given genus, the semi-stability of any syzygy point of the balanced ribbon 
can in principle be verified numerically by enumerating all the states and checking that their convex hull contains 
the trivial state.  We did calculations in \texttt{Macaulay2} and \texttt{polymake} \cite{Macaulay2, polymake} 
that established GIT semi-stability of the $1^{st}$ syzygy point of the balanced ribbon for $g=7,9,11,13$ and the $2^{nd}$ syzygy point for $g=9,11$. (Computations for higher genera appear to be infeasible.)
The main theorem of this paper (on first syzygies) and these calculations on second syzygies (for small genus) 
provide the first evidence for Keel's approach to constructing the canonical model of $\M_g$.

\bibliographystyle{alpha}
\bibliography{references20}

\end{document}